\documentclass[12pt]{amsart}
\usepackage{graphicx} 

\usepackage{amsaddr}
\usepackage{amscd,amssymb}
\usepackage[all]{xy}
\usepackage{a4wide}
\usepackage{enumitem}
\usepackage{graphics,graphicx}
\setlist[itemize]{leftmargin=*}
\setlist[enumerate]{leftmargin=*}
\setlist[description]{leftmargin=*}

\newtheorem{thm}{Theorem}
\newtheorem*{maschkestheorem*}{Maschke's theorem}
\newtheorem*{question*}{Question}
\newtheorem{prop}[thm]{Proposition}
\newtheorem{lem}[thm]{Lemma}

\theoremstyle{definition}
\newtheorem{defi}[thm]{Definition}
\newtheorem{exa}[thm]{Example}
\newtheorem{rem}[thm]{Remark}

\newtheorem*{conventions*}{Conventions}

\newcommand{\N}{\mathbb{N}}

\newcommand{\Set}{{\bf Set}}

\newcommand{\Mod}{{\rm Mod}}
\newcommand{\FMod}{{\rm FMod}}
\newcommand{\SMod}{{\rm SMod}}
\newcommand{\UMod}{{\rm UMod}}

\newcommand{\Grp}{{\bf Grp}}

\newcommand{\id}{{\rm id}}

\newcommand{\Hom}{{\rm Hom}}
\newcommand{\Res}{{\rm Res}}
\newcommand{\Ind}{{\rm Ind}}

\DeclareMathOperator{\C}{\mathsf C}
\DeclareMathOperator{\D}{\mathsf D}
\DeclareMathOperator{\J}{\mathsf J}
\DeclareMathOperator{\HH}{\mathsf H}
\DeclareMathOperator{\E}{\mathsf E}

\DeclareMathOperator{\Sub}{Sub}
\DeclareMathOperator{\Quot}{Quot}
\DeclareMathOperator{\Ker}{Ker}
\DeclareMathOperator{\Ima}{Im}
\DeclareMathOperator{\F}{\mathsf F}
\DeclareMathOperator{\G}{\mathsf G}

\DeclareMathOperator{\R}{\mathsf R}

\title{Separable functors and firm modules}

\author{Patrik Lundström}
\address{Department of Engineering Science,
University West, SE-46186 Trollh\"{a}ttan, Sweden}

\begin{document}

\begin{abstract}
We develop a theory of separable ring extensions 
and separable functors 
for nonunital rings in the setting of firm modules. 
We prove nonunital analogues of classical results 
on functorial separability 
and semisimplicity, and apply these results 
to obtain a locally unital version 
of Maschke's theorem for group rings.
\end{abstract}


\maketitle

\section{Introduction}\label{sec:introduction}

Let $A$ be a ring. By this we mean that 
$A$ is associative, but not necessarily 
unital.
Suppose that $B$ is another ring and 
$f : B \to A$ is a ring homomorphism.
In that case, we say that $A$ is a 
ring extension of $B$ 
and we indicate this by writing $A/B$.

Recall that $A/B$ 
is said to be separable 
if the multiplication map 
$\mu : A \otimes_B A \to A$, defined by
the additive extension of
$\mu(a \otimes a') = aa'$, for $a,a' \in A$,
has a section in the category of 
$A$-bimodules, that is, if
there is an $A$-bimodule map
$\sigma : A \to A \otimes_B A$ such that
$\mu \circ \sigma = \id_A$. 
Here, the $B$-bimodule structure on $A$
is defined via $f$.
Separable ring extensions generalize the classical notion of 
separable algebras over fields, which in turn extends 
separability of field extensions 
(see, e.g., \cite{wisbauer2016} and the references therein).

Separable ring extensions have been studied
by numerous authors (see, e.g., 
\cite{brzezinski2005,CLP,demeyer1971,
haefner2000,iglesias1998,kadison1999,
lundstrom2005,lundstrom2006,
miyashita1971,
nastasescu1989,nystedt2018,
rafael1990,TheohariApostolidi}).
One reason for the sustained interest in these extensions 
is that important properties of the ground ring,
such as semisimplicity,
are often inherited by the larger ring.
Perhaps the most classical example of this
is the following result.

\begin{maschkestheorem*}
Let $G$ be a finite group of 
order $|G|$. Let $B$ be a semisimple 
unital ring with $|G|$ invertible in $B$.
Then the group ring 
$A = B[G]$ is semisimple.
\end{maschkestheorem*}

Classically, this result is proved by a direct 
semisimplicity argument
(see, e.g., \cite[Thm.~(6.1)]{lam2001} or
Maschke's original proof \cite{maschke1899}). 
However, separable ring extensions provide a more conceptual explanation of Maschke’s theorem. 
In \cite{nastasescu1989} (see Definition~\ref{def:separablefunctor}), 
N{\v a}st{\v a}sescu, Van den Bergh and Van Oystaeyen introduce 
the notion of a separable functor. 
Such a functor reflects splittings: 
if a morphism or exact sequence splits after applying the functor, 
then it already splits beforehand. 
In the same article 
(\cite[Prop.~1.3]{nastasescu1989}), 
they prove the following result.

\begin{thm}\label{thm:main1}
Let $f : B \to A$ be a
homomorphism of unital rings.
Then $A/B$ is separable if and only if
the restriction functor
$\Res : {}_A \Mod \to 
{}_B \Mod$ is separable.
\end{thm}

Specializing to group rings, by
\cite[p.~41]{demeyer1971}, we have:

\begin{thm}\label{thm:main2}
Let $G$ be a finite group 
of order $|G|$.
Let $B$ be a unital ring such that 
$|G|$ is invertible in $B$. 
Let $A$ denote the group ring $B[G]$. 
Then $A/B$ is separable.
\end{thm}

By Theorem~\ref{thm:main1}, separability of 
$A/B$ is equivalent to separability 
of the restriction functor
$\Res : {}_A \Mod \to {}_B \Mod$.
Consequently, every exact sequence of $A$-modules 
that splits as a sequence of $B$-modules 
already splits $A$-linearly.
If $B$ is semisimple,
separability of $A/B$ then implies that $A$ is semisimple,
thereby recovering Maschke’s theorem as a consequence of 
separability and functorial splitting.

The rings and modules considered in the classical situation 
above are unital. However, many important algebraic objects 
are defined as infinite direct sums or arise from local constructions, 
and therefore do not in general possess a global unit.
This naturally raises  
the following question: 
\begin{quote}
\emph{Is it possible to extend the classical separable functorial machinery from the unital setting to appropriate categories of modules over nonunital rings?}
\end{quote}

To our knowledge, there appear to be very few results on 
separability of nonunital ring extensions. 
The only references we have found are 
\cite{brzezinski2005}, where Brzeziński, Kadison, and Wisbauer 
study connections between separability of 
\(A\)-rings and \(A\)-corings, 
and \cite{bohm2013}, where 
Böhm and Gómez-Torrecillas 
study splittings of the comultiplication map 
in the context of coalgebras.

Apart from these works, we obtained in 
\cite[Proposition~26]{CLP} a nonunital version of 
Theorem~\ref{thm:main1}
under the assumption that
\(A\) and \(B\) are rings with 
enough idempotents.
Recall from Fuller~\cite{fuller1976} that 
a ring \(A\)
has enough idem\-potents if there exists a 
set \(\{e_i\}_{i\in I}\) of 
pairwise orthogonal idempotents in \(A\)
such that
\(A=\oplus_{i\in I}Ae_i=\oplus_{i\in I}e_iA\).
The module categories considered in 
\cite{CLP}
are the categories \({}_A\UMod\) of 
unitary
left \(A\)-modules, that is, modules 
\(M\) satisfying \(AM=M\), 
where \(A\) is assumed to be a ring 
with enough idempotents.

In the present article, we extend the 
separability result of \cite{CLP}
to the more general setting of 
firm modules introduced by Quillen \cite{quillen1997} 
(see Theorem~\ref{thm:main3}).
Recall that if \(A\) is a ring and \(M\) is a 
left \(A\)-module, then \(M\) is said to 
be firm if the multiplication map
$A \otimes_A M \to M$, induced by the
additive extension of
$A \otimes_A M \ni a \otimes m \mapsto am \in M$,
for $a \in A$ and $m \in M$,
is an isomorphism.
We denote the category of such modules by 
\({}_A \FMod\).
The ring $A$ is said to be firm if it is 
firm as a left $A$-module.
Following Mar{\'\i}n and Laan \cite{marin2023},
we say that a ring homomorphism 
\(f : B \to A\) is left firm if \(A\), 
considered as a left \(B\)-module via~\(f\), 
is firm.
We prove the following firm version of 
Theorem~\ref{thm:main1}:

\begin{thm}\label{thm:main3}
Let $A$ be a firm ring and suppose that
\(f : B \to A\) is a left firm 
ring homomorphism. Then \(A/B\) is
separable if and only if the 
restriction functor
\(\Res_f : {}_A \FMod \to {}_B \FMod\) 
is separable.
\end{thm}

We then apply this result to the study of
semisimple modules in the category 
${}_A \FMod$
when $A$ is a left s-unital ring
(see Theorem~\ref{thm:main4}).
Recall from Tominaga~\cite{tominaga1976}
that a left $A$-module $M$ is 
said to be s-unital, if
for every $m \in M$, one has $m \in Am$.
A ring is said to be left s-unital if it is s-unital
as a left module over itself. 
We say that a ring homomorphism $f : B \to A$ is
left s-unital if $A$, regarded as a left $B$-module via $f$,
is s-unital.
We prove the following s-unital
semisimplicity result:

\begin{thm}\label{thm:main4}
Let $A$ and $B$ be left s-unital rings
with $B$ left semisimple.
Suppose that $f : B \to A$ is a left s-unital 
ring homomorphism such that $A/B$ 
is separable.
Then $A$ is left semisimple.
\end{thm}

This allows us to prove a locally unital version of Maschke's
theorem (see Theorem~\ref{thm:main6}). 
Recall from \'{A}nh and 
M\'{a}rki~\cite{anh1987} that $A$
is called a ring with local units if for each
finite subset $X \subseteq A$ there exists
an idempotent $e \in A$ such that 
$ex = x = xe$ for all $x \in X$.
We first establish the following locally unital
analogue of Theorem~\ref{thm:main2}:

\begin{thm}\label{thm:main5}
Let $G$ be a finite group 
of order $|G|$.
Let $B$ be a ring with local units
satisfying $|G|B = B$. 
Let $A$ denote the group ring $B[G]$. 
Then $A/B$ is separable.
\end{thm}

Combining this with Theorem~\ref{thm:main4}, 
we obtain:

\begin{thm}\label{thm:main6}
Let $G$ be a finite group 
of order $|G|$.
Let $B$ be a semisimple ring with local units 
satisfying $|G|B = B$. 
Then the group ring $A = B[G]$ is semisimple.
\end{thm}

Here is an outline of the article.

In Section~\ref{sec:separablefunctors}, 
we fix our categorical conventions and 
recall the definition of a separable 
functor. We then develop a systematic account of the categorical properties reflected by separable functors. In particular, we show that separability reflects key classes of morphisms, such as monomorphisms, epimorphisms, sections, retractions, and isomorphisms, as well as structural properties of objects, including subobject and quotient simplicity and semisimplicity, and the existence of initial, terminal, and zero objects. 

In Section~\ref{sec:modulecategories}, we introduce our conventions on nonunital rings and module categories, with particular emphasis on unitary, s-unital, and firm modules. We clarify the relationships between these notions and show that they coincide over left s-unital rings. We then study separability of functors between module categories associated to a ring homomorphism $f : B \to A$, focusing on the restriction functor $\Res_f$ and the induction functor $\Ind_f = A \otimes_B -$ on categories of firm modules. 

In Section~\ref{sec:semisimplemodules}, we apply the preceding functorial results to the study of simplicity and semisimplicity in the category ${}_A\FMod$ of firm modules over a left s-unital ring $A$. We show that the usual module-theoretic notions of simple and semisimple modules coincide with the categorical notions of subobject and quotient simplicity and semisimplicity. These results are then used to prove the nonunital version of Maschke’s theorem for group rings over locally unital rings mentioned above, as well as a corresponding hereditary result, thereby extending classical semisimplicity results to a broad nonunital setting.

\section{Separable functors}\label{sec:separablefunctors}

Throughout, we use the following notation. For a set $X$, we let $|X|$ denote its cardinality, and we let $\N$ denote the set of positive integers.
We let $\C$ denote a category. By this we mean that $\C$ consists of a class of objects $\C_0$, a class of morphisms $\C_1$, domain and codomain maps 
$d,c : \C_1 \to \C_0$, and a partially defined associative composition
$\C_1 \times \C_1 \ni (g,f) \mapsto gf \in \C_1$,
defined whenever $d(g)=c(f)$. For each $M \in \C_0$, there is an identity morphism $\id_M \in \C_1$ satisfying $d(\id_M)=c(\id_M)=M$, 
$f\id_M=f$ whenever $d(f)=M$, and 
$\id_N f=f$ whenever $c(f)=N$.
If $f$ is a morphism in $\C$ with $d(f)=M$ and $c(f)=N$, 
then we write $f:M\to N$, and we denote the class of all such morphisms by $\Hom_{\C}(M,N)$.

\begin{defi}\label{def:separablefunctor}
Let $\F : \C \to \D$ be a functor, 
always assumed to be covariant.
Following \cite{nastasescu1989}, 
we say that 
$\F$ is separable
if for all $M,N \in \C_0$, there is a 
function
\begin{itemize}[widest=(SF2)]

\item[(SF0)] 
$\R_{M,N} : \Hom_{\D}(\F(M),\F(N)) \to 
\Hom_{\C}(M,N)$ such that

\item[(SF1)] $\R_{M,N} ( \F(f) ) = f$, for
$f : M \to N$ in $\C$, and

\item[(SF2)] $\F(g) f' = g' \F(f) \Rightarrow
g \R_{M,N}(f') = \R_{M',N'}(g') f$,
for $f : M \to M'$, $g : N \to N'$ in $\C$, and
$f' : \F(M) \to \F(N)$, $g' : \F(M') \to \F(N')$ in $\D$.
\end{itemize}
At times, we indicate the dependence on $\F$ by 
writing $\R_{M,N}^{\F}$ instead of $\R_{M,N}$.
\end{defi} 

\begin{prop}\label{prop:RFid}
Suppose that $\F : \C \to \D$ is a separable
functor. Let $M,N \in \C_0$.

\begin{itemize}[widest=(a)]

\item[{\rm (a)}] $\R_{M,M}(\id_{\F(M)}) = \id_{M}$;

\item[{\rm (b)}] $\Hom_{\C}(M,N) \neq \emptyset$
$\Longleftrightarrow$
$\Hom_{\D}(\F(M),\F(N)) \neq 
\emptyset$;

\item[{\rm (c)}] $\F$ is faithful.

\end{itemize}
\end{prop}

\begin{proof}
For (a) see \cite[p. 399]{nastasescu1989}; (b) and (c)
are immediate.
\end{proof}

The next result shows that the assignment $\R$ 
behaves almost like a functor,
in the sense that the naturality condition
(SF2) can be replaced by (SF3) below.

\begin{prop}\label{prop:newsep}
Let $\F : \C \to \D$ be a functor. 
Then $\F$ is 
separable if and only if for all 
$M,N \in \C_0$, there is a map $\R_{M,N}$ 
satisfying (SF0), (SF1), and 
\begin{itemize}[widest=(SF2)]

\item[{\rm (SF3)}] 
$\R_{M,P}(gf) = 
\R_{N,P}(g) \R_{M,N}(f)$ for $M,N,P \in \C_0$,
$f : \F(M)  \to  \F(N)$, 
\linebreak
$g :  \F(N)  \to  \F(P)$, 
whenever
$f \in \F( \Hom_{\C}(M,N) )$ or 
$g \in \F( \Hom_{\C}(N,P))$.
\end{itemize}
\end{prop}

\begin{proof}
This is essentially \cite[Lemma~1.1]{rafael1990}.
\end{proof}

\begin{prop}\label{prop:nastasescu}
Let $\F : \C \to \D$ and $\G : \D \to \E$
be functors. Put $\HH := \G \F$.

\begin{itemize}[widest=(a)]

\item[{\rm (a)}] 
$\F$ and $\G$ separable
$\Longrightarrow$ $\HH$ separable; 

\item[{\rm (b)}]
$\HH$ separable $\Longrightarrow$ 
$\F$ separable.

\end{itemize}
\end{prop}

\begin{proof}
This is \cite[Lemma 1.1]{nastasescu1989}.
\end{proof}

The following result is probably well known. 
Since we were unable to find an appropriate reference, 
we include a proof.

\begin{lem}\label{lem:limitsreflect}
Every separable functor reflects limits and colimits.
\end{lem}

\begin{proof}
Let \( \F : \C \to \D \) be a 
separable functor.
Let \( \G : \J \to \C \) be a diagram, and let 
\( (\lambda_j : L \to C_j)_{j \in \J} \) 
be a cone in $\C$ such that 
\( ( \F(\lambda_j) : \F(L) \to \F(C_j))_{j \in \J} \) 
is a limit cone of \( \F \circ \G \) in \(\mathcal \D \). 
Given any cone \( (\alpha_j : X \to C_j)_{j \in \J}\) 
in \( \C \), there exists a unique morphism
$h : \F(X) \to \F(L)$ with
$\F( \lambda_j ) h = \F(\alpha_j)$
for all $j \in \J$. 
Put \(g := \R_{X,L}(h)\). Then, for each \(j \in \J\), 
(SF3) and (SF1) imply that
$\lambda_j  g
= \lambda_j \R_{X,L}(h)
= \R_{X,C_j}(\F(\lambda_j)  h)
= \R_{X,C_j}(\F(\alpha_j))
= \alpha_j$.
Thus \(g\) is a mediating morphism.
If \(g' : X \to L\) is another such morphism, then
$\F(\lambda_j)  \F(g')
= \F(\alpha_j)
= \F(\lambda_j)  \F(g)$
for all \(j \in \J\), so by the uniqueness 
of \(h\) we get \(\F(g') = \F(g)\).
By Proposition~\ref{prop:RFid}\textup{(c)}, \(g' = g\).
Thus \( (\lambda_j : L \to C_j)_{j \in \J} \)
is a limit cone in \( \C \). Hence, \( \F \) reflects limits.
The proof for colimits is dual. 
\end{proof}

Let $f : M \to N$ be a morphism in $\C$.
Then $f$ is called 
a monomorphism if for all morphisms 
$g_1, g_2 : P \to M$, the equality 
$f g_1 = f g_2$ implies $g_1 = g_2$. 
Dually, $f$ is called an 
epimorphism if for all morphisms 
$h_1, h_2 : N \to Q$, the equality 
$h_1 f = h_2 f$ implies $h_1 = h_2$.

\begin{prop}
Suppose that $\F : \C \to \D$ is a separable
functor. Let $f \in \C_1$. 
\begin{itemize}[widest=(a)]

\item[{\rm (a)}] $\F(f)$ monomorphism
$\Longrightarrow$ $f$ monomorphism;

\item[{\rm (b)}] $\F(f)$ epimorphism
$\Longrightarrow$ $f$ epimorphism.

\end{itemize}
\end{prop}

\begin{proof}
This follows from Lemma~\ref{lem:limitsreflect}.
Indeed, a morphism is a monomorphism if and only if 
$(\id,\id)$ is the pullback of the corresponding pair $(f,f)$.
Thus (a) holds. The statement in (b) is dual.
For more details, see e.g. 
\cite[p. 72, Exercise 4]{maclane1998}).
\end{proof}

Let $f : M \to N$ and $g : N \to M$ be morphisms 
in $\C$ such that
$fg = \id_N$.  Then $f$ is called 
a retraction of $g$, and
$g$ is called a section of $f$. 
If also $gf = \id_M$, 
then $f$ is said to be an isomorphism 
with inverse $g$. 

\begin{prop}\label{prop:retractionsection}
Suppose that $\F : \C \to \D$ is a separable
functor and let $f : M \to N$ be a morphism in $\C$.
\begin{itemize}[widest=(a)]

\item[{\rm (a)}] $r : \F(N) \to \F(M)$
a retraction of $\F(f)$ $\Longrightarrow$
$\R_{N,M}(r)$ a retraction of $f$;

\item[{\rm (b)}] $s : \F(N) \to \F(M)$  
a section of $\F(f)$ $\Longrightarrow$ 
$\R_{N,M}(s)$ a section of $f$;

\item[{\rm (c)}] $g : \F(N) \to \F(M)$ 
an inverse of $\F(f)$
$\Longrightarrow$ $\R_{N,M}(g)$ an inverse of $f$.

\end{itemize}
\end{prop}

\begin{proof}
This is \cite[Prop. 1.2.1]{nastasescu1989}.
\end{proof}

Let $M \in \C_0$. 
Consider the class of all monomorphisms
in $\C$ with codomain $M$.  
We define an equivalence relation $\sim_M$ on this 
class as follows.
Given monomorphisms $f : P \to M$ and $g : Q \to M$, 
we write $f \sim_M g$ if there exists an isomorphism $h : P \to Q$ with $f = g h$.
The equivalence class $[f]_{\sim_M}$ 
of such a 
monomorphism $f : P \to M$ is called a 
subobject of $M$.  
We denote the collection of subobjects of $M$ by 
$\Sub_{\C}(M)$.

\begin{lem}\label{lem:simequiv}
Suppose that $f : P \to M$ and $g : Q \to M$ are 
monomorphisms in $\C$. Then $[f]_{\sim_M} = 
[g]_{\sim_M}$ if and only if there are 
morphisms $p : P \to Q$ and $q : Q \to P$
such that $f = gp$ and $g = f q$.
\end{lem}

\begin{proof}
See the discussion in \cite[p.~122]{maclane1998}.
\end{proof}

Let $\F : \C \to \D$ be a functor
that preserves monomorphisms.
Take $M \in \C_0$.
Define $\alpha_{M} : 
\Sub_{\C}(M) \to \Sub_{\D}(\F(M))$
by $\alpha_{M}( [ m ]_{\sim_M} ) = 
[\F(m)]_{\sim_{\F(M)}}$, for 
monomorphisms $m : P \to M$.

\begin{prop}\label{prop:injective}
Let $\F : \C \to \D$ be a functor
that preserves monomorphisms. Let $M \in \C_0$. 
Then $\alpha_{M}$ is well-defined.
If $\F$ is also separable, then
$\alpha_{M}$ is injective.
\end{prop}

\begin{proof}
Let $f : P \to M$ and $g : Q \to M$ be monomorphisms 
in $\C$ such that $f = gh$ for some isomorphism 
$h : P \to Q$. Then $\F(f)=\F(g)\F(h)$.
%Since $\F$ preserves monomorphisms, 
By the assumptions,
$\F(f)$ and $\F(g)$ are monomorphisms. 
%and since $h$ is an isomorphism, $\F(h)$ is an isomorphism. 
Hence $[\F(f)]_{\sim_{\F(M)}} = [\F(g)]_{\sim_{\F(M)}}$.
Thus, $\alpha_M$ is well-defined.

Suppose now that $\F$ is separable and 
$[\F(f)]_{\sim_{\F(M)}} = [\F(g)]_{\sim_{\F(M)}}$.
By Lemma~\ref{lem:simequiv}, there exist morphisms
$p : \F(P) \to \F(Q)$ and $q : \F(Q) \to \F(P)$ with
$\F(f) = \F(g)p$ and $\F(g) = \F(f)q$.
By (SF1) and (SF3),
$f =\R_{P,M}(\F(f))
= \R_{P,M}(\F(g)p)
= \R_{Q,M}(\F(g)) \R_{P,Q}(p)
= g \R_{P,Q}(p)$,
and similarly,
$g = \R_{Q,M}(\F(g))
= \R_{Q,M}(\F(f)q)
= \R_{P,M}(\F(f)) \R_{Q,P}(q)
= f \R_{Q,P}(q)$.
Hence, by Lemma~\ref{lem:simequiv},
$[f]_{\sim_M} = [g]_{\sim_M}$.
%Therefore, $\alpha_M$ is injective.
\end{proof}

We introduce the 
following terminology.
We say that $M \in \C_0$ is
subobject trivial if 
$|\Sub_{\C}(M)| = 1$,
subobject nontrivial if 
$|\Sub_{\C}(M)| \geq 2$, and
subobject simple if 
$|\Sub_{\C}(M)|=2$.

\begin{prop}\label{prop:subobjectsimple}
Suppose that $\F : \C \to \D$ is a separable
functor which preserves monomorphisms. 
Let $M \in \C_0$ be subobject nontrivial
and $\F(M)$ subobject simple in $\D$. Then
$M$ is subobject simple in $\C$.
\end{prop}

\begin{proof}
Since $\F(M)$ is subobject simple in $\D$,
$|\Sub_{\D}(\F(M))|=2$.
The map $\alpha_M$ is injective,
by Proposition \ref{prop:injective}.
Thus, $|\Sub_{\C}(M)| \leq 2$. 
On the other hand, since
$M$ is subobject nontrivial, 
$|\Sub_{\C}(M)| \geq 2$.
Therefore $|\Sub_{\C}(M)|=2$.
\end{proof}

Let $M \in \C_0$.
We say that $M$ is subobject semisimple
if every monomorphism $m: P \to M$ 
has a retraction $r : M \to P$.

\begin{prop}\label{prop:Fsemisimple}
Suppose that $\F : \C \to \D$ is a separable
functor which preserves monomorphisms. 
Let $M \in \C_0$ have the property that
$\F(M)$ is subobject semisimple in $\D$. Then
$M$ is subobject semisimple in $\C$.
\end{prop}

\begin{proof}
Suppose that $f : P \to M$ is a monomorphism in $\C$.
Since $\F$ preserves monomorphisms,
$\F(f) : \F(P) \to \F(M)$ is a monomorphism in $\D$.
Because $\F(M)$ is subobject semisimple, 
$\F(f)$ has a retraction $r : \F(M) \to \F(P)$. 
By 
Proposition \ref{prop:retractionsection}(a),
$\R_{M,P}(r)$ is a retraction of $f$. 
Hence, $M$ is subobject semisimple in $\C$.
\end{proof}

\begin{exa}\label{ex:simplenotsemisimple}
Let $\Set$ denote the category of sets.

(a) In $\Set$, the monomorphisms are
precisely the injective maps.
Using this, it is easy to see that
all sets are subobject semisimple. 
All singleton sets $\{ * \}$ are subobject
simple since, up to equivalence
there are precisely two injections into
$\{ * \}$ namely $\id_{ \{ * \} }$ and
the empty function $\emptyset : 
\emptyset \to \{ * \}$. It is not hard
to see that the singleton sets in fact
are all subobject simple sets.
Therefore, in $\Set$, every
subobject simple object is trivially
subobject semisimple.

(b) In an arbitrary category, a subobject simple object need not be subobject semisimple.
Indeed, consider
the category whose objects are $M$ and $N$, and 
whose only morphisms are $\id_M,\id_N$ and
$f : N \to M$. Subobjects of $M$
correspond to the monomorphisms into $M$,
that is $\id_M$ and $f$. Therefore,
$M$ is subobject simple. But 
$f$ does not have a retraction.
Hence, $M$ is not subobject semisimple.
\end{exa}

Let $M \in \C_0$. 
Consider the class of all epimorphisms
in $\C$ with domain $M$.  
We define an equivalence relation $\approx_M$ on 
this class as follows.
Given epimorphisms $f : M \to P$ and $g : M \to Q$, we write
$f \approx_M g$ if there exists an isomorphism $h : 
Q \to P$ with $f = h g$.
The equivalence class $[f]_{\approx_M}$ of such an 
epimorphism $f : M \to P$ is called a 
quotient object of $M$.  
We denote the collection of quotient objects of $M$ by $\Quot_{\C}(M)$.

\begin{lem}\label{lem:approx}
Suppose that $f : M \to P$ and $g : M \to Q$ are 
epimorphisms in $\C$. Then $[f]_{\approx_M} = 
[g]_{\approx_M}$ if and only if there are 
morphisms $p : Q \to P$ and $q : P \to Q$
such that $f = pg$ and $g = qf$.
\end{lem}

\begin{proof}
Dual to Lemma~\ref{lem:simequiv}.
\end{proof}

Let $\F : \C \to \D$ be a functor
that preserves epimorphisms.
Take $M \in \C_0$. Define
$\beta_{M} : 
\Quot_{\C}(M) \to \Quot_{\D}(\F(M))$ 
by  $\beta_{M}( [ e ]_{\approx_M} ) = 
[\F(e)]_{\approx_{\F(M)}}$,
for epimorphisms $e : M \to P$.

\begin{prop}\label{prop:injectiveagain}
Let $\F : \C \to \D$ be a functor
that preserves epimorphisms. Then
$\beta_{M}$ is well-defined.
If $\F$ is also separable, then
$\beta_{M}$ is injective.
\end{prop}

\begin{proof}
Dual to Proposition~\ref{prop:injective}, using Proposition~\ref{prop:retractionsection}(b). % in place of Proposition~\ref{prop:retractionsection}(a).
\end{proof}

We say that an object $M \in \C_0$ is
quotient trivial if 
$|\Quot_{\C}(M)| = 1$,
quotient nontrivial if 
$|\Quot_{\C}(M)| \geq 2$, and
quotient simple if 
$|\Quot_{\C}(M)|=2$.

\begin{prop}
Suppose that $\F : \C \to \D$ is a separable
functor which preserves epimorphisms. 
Let $M \in \C_0$ be quotient nontrivial
and $\F(M)$ quotient simple in $\D$. Then
$M$ is quotient simple in $\C$.
\end{prop}

\begin{proof}
Dual to Proposition \ref{prop:subobjectsimple}, 
using Proposition \ref{prop:injectiveagain}. 
%in place of Proposition \ref{prop:injective}.
\end{proof}

Let $M \in \C_0$.
We say that $M$ is quotient semisimple
if every epimorphism $e : M \to P$ has a 
section $s : P \to M$.

\begin{prop}
Suppose that $\F : \C \to \D$ is a separable
functor which preserves epimorphisms. 
Let $M \in \C_0$ have the property that
$\F(M)$ is quotient semisimple in $\D$. Then
$M$ is quotient semisimple in $\C$.
\end{prop}

\begin{proof}
Dual to Proposition~\ref{prop:Fsemisimple}, using Proposition~\ref{prop:retractionsection}(b). % in place of Proposition~\ref{prop:retractionsection}(a).
\end{proof}

\begin{exa}
(a) In $\Set$, the epimorphisms are
precisely the surjective maps.
Using this, it is easy to see that
all sets are quotient semisimple. 
All two-element sets are quotient simple since, up to equivalence, there are precisely two surjections from such a set: the identity map and the unique surjection onto a singleton.
One can easily show that sets 
of other cardinalities are not
quotient simple.
Thus, in $\Set$, every quotient simple
object is trivially quotient semisimple.

(b) By taking the dual of the category in
Example~\ref{ex:simplenotsemisimple}(b), 
we see that for arbitrary categories,
a quotient simple object is not  
necessarily quotient semisimple.
\end{exa}

Recall that $I \in \C_0$ is called 
initial if
$|\Hom_{\C}(I,M)|=1$ for all $M \in \C_0$. 
Dually, \(T \in \C_0\) is called 
terminal if $|\Hom_{\C}(M,T)|=1$,
for all $M \in \C_0$.
An object that is 
both initial and terminal is called a 
zero object. 
A category with a zero object is called 
pointed.

\begin{prop}
Suppose that $\F : \C \to \D$ is a separable
functor. Let $I,Z,T \in \C_0$. 
\begin{itemize}[widest=(a)]

\item[{\rm (a)}] $\F(I)$ initial
in $\D$ $\Longrightarrow$ $I$ initial in $\C$;

\item[{\rm (b)}] $\F(T)$ terminal
in $\D$ $\Longrightarrow$ 
$T$ terminal in $\C$;

\item[{\rm (c)}] $\F(Z)$ a 
zero object in $\D$ $\Longrightarrow$ 
$Z$ a zero object in $\C$.

\end{itemize}
\end{prop}

\begin{proof}
This follows from Lemma~\ref{lem:limitsreflect}, 
since terminal objects are limits of the empty diagram 
and initial objects are colimits of the empty diagram.
\end{proof}

Let $\C$ be a pointed category with a
zero object $Z \in \C_0$.
Take \(M,N \in \C_0\). 
We denote the unique morphisms 
\(Z \to M\) and \(M \to Z \) by \(0_{Z,M}\) 
and \(0_{M,Z}\), respectively, and
we set $0_{M,N} := 0_{Z,N} 0_{M,Z}$.
Clearly, $0_{Z,N}$ is a monomorphism and 
$0_{M,Z}$ is an epimorphism.

\begin{lem}\label{lem:pointed}
Suppose that $\C$ is a pointed category 
with a zero object $Z$. Let $M \in \C_0$.
The following statements are equivalent:
\begin{itemize}[widest=(iii)]

\item[{\rm (i)}] $M$ is a zero object in $\C$; 
\quad
{\rm (ii)} there is an isomorphism $Z \to M$; 

\item[{\rm (iii)}] $0_{Z,M} \sim_M \id_M$;
\quad
{\rm (iv)} $M$ is subobject trivial;

\item[{\rm (v)}] $0_{M,Z} \approx_M \id_M$;
\quad
{\rm (vi)} $M$ is quotient trivial.

\end{itemize}
\end{lem}

\begin{proof}
It suffices to prove (i)$\Rightarrow$(iii)$\Rightarrow$(iv)$\Rightarrow$(i), because the equivalence (i)$\Leftrightarrow$(ii) is well known, and the circle (i)$\Rightarrow$(v)$\Rightarrow$(vi)$\Rightarrow$(i) is dual.

(i)$\Rightarrow$(iii):
Suppose that $M$ is a zero object. Then it follows that
$0_{Z,M}:Z\to M$ is an isomorphism, so 
$0_{Z,M}=\id_M\,0_{Z,M}$. Hence $0_{Z,M}\sim_M \id_M$.

(iii)$\Rightarrow$(iv):
Suppose that $0_{Z,M}\sim_M \id_M$.
Let $h:P\to M$ be a monomorphism.
Choose an isomorphism $g:M\to Z$ such that
$\id_M=0_{Z,M}g$.
Since $Z$ is a zero object, $g=0_{M,Z}$. Thus
$h=\id_M h=0_{Z,M}0_{M,Z}h$.
Now $0_{M,Z}h:P\to Z$, so by uniqueness of morphisms into $Z$,
$0_{M,Z}h=0_{P,Z}$. Hence
$h=0_{Z,M}0_{P,Z}$. Also,
$h0_{Z,P}=0_{Z,M}0_{P,Z}0_{Z,P}=0_{Z,M}$,
since $0_{P,Z}0_{Z,P}=\id_Z$.
Therefore, by Lemma~\ref{lem:simequiv}, \(h\sim_M 0_{Z,M}\).
Thus every subobject of \(M\) equals \([0_{Z,M}]_{\sim_M}\), so \(M\) is subobject trivial.

(iv)$\Rightarrow$(i):
Suppose that $M$ is subobject trivial.
$0_{Z,M}\sim_M \id_M$.
Hence there exists an isomorphism \(f:M\to Z\) such that
$\id_M=0_{Z,M}f$.
Therefore \(0_{Z,M}\) is an isomorphism. Since \(Z\) is a zero object, it follows that \(M\) is a zero object.
\end{proof}

\begin{prop}\label{prop:simpleimpliessemisimple}
Let $\C$ be a pointed category
with a zero object $Z$.
Let $M \in \C_0$.
\begin{itemize}[widest=(a)]

\item[{\rm (a)}] $M$ subobject simple
$\Longrightarrow$ $M$ subobject semisimple;

\item[{\rm (b)}] $M$ quotient simple
$\Longrightarrow$ $M$ quotient semisimple.

\end{itemize}
\end{prop}

\begin{proof}
We prove (a). The proof of (b) is dual, using Lemma~\ref{lem:pointed}\textup{(v)} in place of Lemma~\ref{lem:pointed}\textup{(iii)}.

Suppose that $M$ is subobject simple.
By Lemma \ref{lem:pointed}(iii), 
$\Sub_{\C}(M)$ consists of the two distinct
classes $[0_{Z,M}]_{\sim_M}$ and
$[\id_M]_{\sim_M}$.
Let $m : P \to M$ be a monomorphism.
We now show that $m$ has a retraction.
To this end, we consider two cases.

Case 1: $[m]_{\sim_M} = [0_{Z,M}]_{\sim_M}$.
Then there is an isomorphism
$f : P \to Z$ with $m = 0_{Z,M} f$. 
But since $0_{P,Z}$ is the unique
morphism $P \to Z$ it follows that $f=0_{P,Z}$
is an isomorphism and 
$m = 0_{Z,M} 0_{P,Z} = 0_{P,M}$.
Put $r := 0_{M,P}$. Then $rm : P \to P$.
By Lemma \ref{lem:pointed}(i),
$P$ is a zero object of $\C$. 
In particular, 
$\id_P$ is the unique morphism $P \to P$,
so $rm = \id_P$.

Case 2: $[m]_{\sim_M} = [\id_M]_{\sim_M}$.
Then there is an isomorphism
$f : P \to M$ with $m = \id_M f = f$.
But then $rm = \id_P$ for $r:= f^{-1}$. 
\end{proof}

\begin{exa}
Let $\Grp$ denote the category of groups.
It is well known that this category is
pointed with the trivial group as a 
zero object. Let $G$ denote a group.

(a) In $\Grp$, monomorphisms are exactly 
injective homomorphisms, and two 
monomorphisms into $G$ represent the 
same subobject of $G$ if and only if 
they have the same image in $G$.
Therefore, the elements of $\Sub_{\Grp}(G)$
correspond to the set of subgroups of $G$.
By Cauchy's theorem 
(see~\cite[Thm. 11, p.~93]{dummit2004}), 
$G$ is subobject simple precisely when
$G$ is a cyclic group of prime order. 
By a result by 
Baer~\cite[Thm. 3]{baer1946}, $G$ is
subobject semisimple if and only if
$G$ is an abelian group all of whose elements have
finite square-free order. 
Note that Baer~\cite{baer1946}
does not use our categorical terminology.

(b) In $\Grp$, epimorphisms are precisely
surjective homomorphisms, and the quotient
object of $G$ represented by an epimorphism 
$q : G \to Q$ corresponds to the quotient group 
$G / \ker(q)$. So quotient objects of $G$ correspond 
to isomorphism classes of quotient 
groups of $G$.
A group $G$ is therefore 
quotient simple in the categorical
sense if and only if $G$ has no nontrivial
proper normal subgroups. 
Thus, quotient simple groups 
are exactly the simple groups in the
classical sense.
Furthermore, $G$ is 
quotient semisimple if and 
only if every normal subgroup of $G$ is 
complemented. 
Indeed, an epimorphism $G \to Q$ 
splits precisely when its kernel $N$ 
has a complement $H$ in $G$ with 
$G = NH$ and $N \cap H = 1$, so that 
$G \cong N \rtimes H$. 
Thus quotient semisimple groups are exactly 
the $nC$-groups studied in the 
group-theoretic literature
(see, e.g., \cite{christensen1967}).
To the author's knowledge, no complete classification
of $nC$-groups is known.
\end{exa}

Recall that $P \in \C_0$ is called
projective if for every epimorphism
$e : M \to N$ and every 
$f : P \to N$, there is a morphism
$g : P \to M$ with $e g = f$. 
Dually, $Q \in \C_0$ is called injective if for every monomorphism 
$m : M \to N$ and every  
$f : M \to Q$, there is a morphism
$g : N \to Q$ with $g m = f$.

\begin{prop}\label{prop:Fseparableprojective}
Suppose that $\F : \C \to \D$ is a separable
functor. Let $P,Q \in \C_0$.
\begin{itemize}[widest=(a)]

\item[{\rm (a)}] $\F(P)$ projective and
$\F$ preserves epimorphisms
$\Longrightarrow$ $P$ projective;

\item[{\rm (b)}] $\F(Q)$ injective and
$\F$ preserves monomorphisms
$\Longrightarrow$ $Q$ injective.

\end{itemize}
\end{prop}

\begin{proof}
This is \cite[Prop. 1.2.2 and Prop. 1.2.3]{nastasescu1989}.
\end{proof}

\section{Module categories}\label{sec:modulecategories}

Let $A$ be a ring. 
By this we mean that $A$ is 
associative, but not necessarily unital.
Let $M$ be a left $A$-module. 
By this we mean
that $M$ is an additive group 
equipped with a biadditive 
map $A \times M \ni (a,m) \mapsto am \in M$
satisfying 
$(aa')m = a(a'm)$ for $a,a' \in A$
and $m \in M$. We let ${}_A \Mod$ denote
the category having left $A$-modules
as objects and $A$-linear maps as
morphisms. Similarly, the category
$\Mod_A$ of right $A$-modules is defined.
Let $B$ be another ring. If $M$ is a
left $A$-module and a right $B$-module, then
$M$ is said to be an $A$-$B$-bimodule if 
$(am)b = a(mb)$ for $a \in A$, $m \in M$ 
and $b \in B$.
We let ${}_A \Mod_B$ denote the category
having $A$-$B$-bimodules as objects and
as morphisms maps that are simultaneously
left $A$-linear and right $B$-linear.
If $M$ is an $A$-$B$-bimodule where
$A = B$, then $M$ is said 
to be an $A$-bimodule.

Suppose that $M$ is a left $A$-module.
Then $M$ is said to be unital
if there is $a \in A$ such that for every
$m \in M$, the equality $am = m$ holds.
We let $AM$
denote the set of all finite sums 
of elements of the form $am$ for $a \in A$
and $m \in M$. 
Following \'{A}nh and 
M\'{a}rki~\cite{anh1987}, 
we say that $M$ is
unitary if $AM = M$.
We let $\mu_{A,M}$ denote the 
multiplication map 
$A \otimes_A M \to M$ defined by 
the additive extension of
$\mu_{A,M}(a \otimes m) = am$,
for $a \in A$ and $m \in M$.
Following Quillen~\cite{quillen1997}, 
we say that $M$ is  
firm if $\mu_{A,M}$ is an 
isomorphism in ${}_A \Mod$.
Finally, following 
Tominaga~\cite{tominaga1976},
we say that $M$ is s-unital if
for every $m \in M$, we have $m \in Am$. 
The ring $A$ is said to be left unital
(s-unital, unitary, firm) if it is 
unital (s-unital, unitary, firm) as a
left $A$-module.
In the sequel, 
we will make use of the following result
by Tominaga \cite[Thm. 1]{tominaga1976}
(see also \cite{nystedt2019}):

\begin{prop}\label{prop:sunitalequivalent}
Suppose that $A$ is a ring and 
$M$ is a left $A$-module.
Then $M$ is s-unital if and only
if for each finite subset $X \subseteq M$
there exists $a \in A$ such that 
$a m = m$ for all $m \in X$.
\end{prop}

We introduce the following full subcategories of \({}_A \Mod\):
\begin{itemize}

\item \({}_A \UMod\), the category of unitary left \(A\)-modules;

\item \({}_A \FMod\), the category of firm left \(A\)-modules;

\item \({}_A \SMod\), the category of s-unital left \(A\)-modules;

\item \({}_A \Mod^1\), the category of 
unital left \(A\)-modules.

\end{itemize}

Clearly, the following implications hold
for these categories:
\begin{equation}\label{eq:implications}
\mbox{unital $\Longrightarrow$ s-unital
$\Longrightarrow$ unitary, \ \ and \ \ 
unital $\Longrightarrow$ firm 
$\Longrightarrow$ unitary.} 
\end{equation} 
By the following examples, none of these implications is reversible.

\begin{exa}\label{ex:examplesrings}
(a) s-unital $\nRightarrow$ unital:
Let \( K \) be a field, and 
consider the ring $K^{(\N)}$ consisting of 
all sequences $(k_n)_{n \in \N}$, where 
$k_n \in K$, and all but finitely many
$k_n = 0$, 
with pointwise addition and multiplication.
Then $K^{(\N)}$ is left s-unital but
not left unital. 

(b) unitary $\nRightarrow$ s-unital:
Consider the ring $C_0(\N,\mathbb{R})$
consisting of all real sequences 
$(a_n)_{n\in\N}$, such that 
$\lim_{n \to \infty} a_n = 0$, 
with pointwise addition and multiplication.
From 
$(a_n)_{n \in \N} =
(\sqrt{|a_n|})_{n \in \N} \cdot 
(\operatorname{sgn}(a_n)
\sqrt{|a_n|})_{n \in \N}$, it follows that
$C_0(\N,\mathbb{R})$ is unitary.
However, this ring is not left s-unital,
which is easily seen by considering the
sequence $b := (1/n)_{n \in \N}$;
there is no element 
$a \in C_0(\N,\mathbb{R})$ such that $ab = b$.

(c) unitary $\nRightarrow$ firm: see 
\cite[Example 1.2]{caenepeel1998}.

(d) firm $\nRightarrow$ unital:
The ring $K^{(\N)}$ from (a) is not 
unital. 
It is easy to see, either by a direct argument or by using Proposition~\ref{prop:sunitalimpliesfirm} below, that it is firm as a left module over itself.
\end{exa}

If we restrict to modules over left s-unital rings, 
then the classes of s-unital modules, firm modules, and unitary modules coincide.

\begin{prop}\label{prop:sunitalimpliesfirm}
Let $A$ be a left s-unital ring.
Then ${}_A \UMod = {}_A \FMod = {}_A \SMod$.
\end{prop}

\begin{proof}
Take $M \in {}_A \Mod$.
By \eqref{eq:implications}, it is enough 
to show $M$ unitary $\Rightarrow$
$M$ s-unital $\Rightarrow$ $M$ firm.
Suppose first that $M$ is unitary.
Take $m \in M$. 
Since $AM=M$ there are $n \in \N$,
$a_1,\ldots,a_n \in A$ and 
$m_1,\ldots,m_n \in M$ with 
$\sum_{i=1}^n a_i m_i = m$.
Since $A$ is left s-unital, there is,
by Proposition~\ref{prop:sunitalequivalent},
$b \in A$ with $b a_i = a_i$,
for all $i \in \{ 1,\ldots,n \}$.
Then $b m = 
\sum_{i=1}^n b a_i m_i = 
\sum_{i=1}^n a_i m_i = m$,
showing that $M$ is s-unital.

Suppose now that $M$ is s-unital.
Surjectivity of $\mu_{A,M}$ follows immediately from the
fact that $M$ is s-unital. Let us show that
$\mu_{A,M}$ is injective. Suppose that $z \in A \otimes_A M$
satisfies $\mu_{A,M}(z)=0$. Take $n \in \N$,
$a_i \in A$ and $x_i \in M$, for $i=1,\ldots,n$,
with $z = \sum_{i=1}^n a_i \otimes x_i$. Then
$\sum_{i=1}^n a_i x_i = \mu_{A,M}(z) = 0$.
Since $A$ is left s-unital, there exists $b \in A$
such that $b a_i = a_i$ for each $i$. Hence,
$z = \sum_{i=1}^n ba_i \otimes x_i = 
\sum_{i=1}^n b \otimes a_i x_i = 
b \otimes \sum_{i=1}^n a_i x_i = b \otimes 0 = 0$.
Therefore, $M$ is firm. 
\end{proof}

\begin{lem}\label{lem:Mfirmifsection}
Suppose that $A$ is a ring and  
$M$ is a left $A$-module.
Then $M$ is firm if and only if 
$\mu_{A,M} : A \otimes_A M \to M$ has 
a section in  ${}_A \Mod$.
\end{lem}

\begin{proof}
The ``only if'' part is trivial.
Now we show the ``if'' part.
Suppose that $\mu_{A,M}$ has a section
$\sigma$ in ${}_A \Mod$. 
Then $\mu_{A,M}$ is surjective. 
It remains to
show that $\mu_{A,M}$ is injective.
Take $z \in \ker(\mu_{A,M})$. Write 
$z = \sum_{i=1}^n a_i \otimes x_i$ for some 
$n \in \N$, $a_i \in A$ and $x_i \in M$.
For each $i$, take $n_i \in \N$, 
$b_{ij} \in A$ and $y_{ij} \in M$ with
$\sigma(x_i) = \sum_{j=1}^{n_i} b_{ij} \otimes y_{ij}$.
Then
\[
\begin{array}{rcl}
z &=& 
\sum_{i=1}^n a_i \otimes x_i 
= \sum_{i=1}^n a_i \otimes 
\mu_{A,M}(\sigma(x_i))  
= \sum_{i=1}^n \sum_{j=1}^{n_i}
a_i \otimes b_{ij} y_{ij} \\
&=& 
\sum_{i=1}^n \sum_{j=1}^{n_i}
a_i b_{ij} \otimes y_{ij} 
= \sum_{i=1}^n a_i \sigma(x_i) 
= \sigma \left( \sum_{i=1}^n a_i x_i \right) \\
&=& 
(\sigma \circ \mu_{A,M}) 
\left( \sum_{i=1}^n a_i \otimes x_i \right) 
= \sigma( \mu_{A,M} (z) ) 
= \sigma(0) = 0,
\end{array} 
\]
showing that $\mu_{A,M}$ is injective.
\end{proof}

Let $f : B \to A$ be a ring
homomorphism, and $M$ a left $A$-module.
We view $M$ as a left $B$-module by 
letting $B$ act on $M$ via $f$, that is, 
for \(b \in B\) and \(m \in M\),
we define $b \cdot m := f(b)m$.
This defines the
restriction functor
$\Res_f : {}_A \Mod \to {}_B \Mod$.
We say that 
$f$ is left unital (respectively 
s-unital, firm, unitary) 
if the left $B$-module
$\Res_f(A)$ is unital 
(respectively s-unital, firm, unitary).

\begin{prop}
Let $f : B \to A$ be a ring homomorphism. 
\begin{itemize}[widest=(a)]

\item[{\rm (a)}] $f$ left unital
$\Longrightarrow$
$\Res_f : {}_A \UMod \to {}_B \Mod^1$;

\item[{\rm (b)}] $f$ left s-unital
$\Longrightarrow$
$\Res_f : {}_A \UMod \to {}_B\SMod$;

\item[{\rm (c)}] $f$ left unitary 
$\Longrightarrow$
$\Res_f : {}_A \UMod \to {}_B\UMod$;

\item[{\rm (d)}] $f$ left firm 
$\Longrightarrow$
$\Res_f : {}_A \FMod \to {}_B\FMod$.

\end{itemize}
\end{prop}

\begin{proof}
Let $M$ be a left $A$-module.

(a) Suppose that $f$ is left unital and
that $M$ is unitary as a left $A$-module. 
Since $A$ is unital as a left $B$-module,
there is $b \in B$ with $f(b)a = a$,
for all $a \in A$.
Take $m \in M$. Since $M$ is unitary
as a left $A$-module, there are $n \in \N$,
$a_1,\ldots,a_n \in A$ and 
$m_1,\ldots,m_n \in M$ with 
$m = \sum_{i=1}^n a_i m_i$. Then
$f(b)m = \sum_{i=1}^n f(b) a_i m_i = 
\sum_{i=1}^n a_i m_i = m$,
showing that $M$ is unital as a 
left $B$-module.

(b) Suppose that $f$ is left s-unital and
that $M$ is unitary as a left $A$-module.
Take $m \in M$. Then there are
$n \in \N$,
$a_1,\ldots,a_n \in A$ and 
$m_1,\ldots,m_n \in M$ such that
$m = \sum_{i=1}^n a_i m_i$.
Since $A$ is s-unital as a left $B$-module,
by Proposition \ref{prop:sunitalequivalent},
there is $b \in B$ such that
$f(b) a_i = a_i$ for all 
$i \in \{ 1,\ldots,n \}$.
Then $f(b) m = 
\sum_{i=1}^n f(b) a_i m_i = 
\sum_{i=1}^n a_i m_i = m$, 
showing that $M$ is s-unital as a
left $B$-module.

(c) Suppose that $f$ is left unitary and
that $M$ unitary as a left $A$-module.
Then $f(B)A = A$ and $AM = M$. Hence,
$f(B)M = f(B)AM = AM = M$, showing that 
$M$ is unitary as a left $B$-module.

(d) Suppose that  $f$ is left firm and
$M$ is firm as a left $A$-module.
By Lemma~\ref{lem:Mfirmifsection}, 
we need to show that $\mu_{B,M}$
has a section in 
${}_B \Mod$. To this end, we first 
note that, by the 
assumptions, $\mu_{B,A}$ has a section
$\sigma_{B,A}$ in ${}_B \Mod$ and 
$\mu_{A,M}$ has a section $\sigma_{A,M}$
in ${}_A \Mod$. Define 
$\sigma : M \to B \otimes_B M$ by
$\sigma := (\id_B \otimes \mu_{A,M})
\circ
(\sigma_{B,A} \otimes \id_M) 
\circ 
\sigma_{A,M}$.
Then $\sigma$, being a composition
of morphisms ${}_B \Mod$, is again a
morphism in ${}_B \Mod$.
We now show that 
$\mu_{B,M} \circ \sigma = \id_M$.
Take $m \in M$,
$p \in \N$, $m_1,\ldots,m_p \in M$
and $a_1,\ldots,a_p \in A$ with 
$\sigma_{A,M}(m) = 
\sum_{i=1}^p a_i \otimes m_i$.
By the assumptions, 
$\sum_{i=1}^p a_i m_i = 
(\mu_{A,M} \circ \sigma_{A,M})(m) = m$.
For each $i \in \{ 1,\ldots,p \}$,
take $q_i \in \N$, 
$a_1^{(i)},\ldots,a_{q_i}^{(i)} \in A$ and
$b_1^{(i)},\ldots,b_{q_i}^{(i)} \in B$ with
$\sigma_{B,A}(a_i) = 
\sum_{j=1}^{q_i} b_j^{(i)}\otimes a_j^{(i)}$.
By the assumptions, for each 
$i \in \{ 1,\ldots,p \}$, we get
$\sum_{j=1}^{q_i} f(b_j^{(i)}) a_j^{(i)} = 
(\mu_{B,A} \circ \sigma_{B,A})(a_i) = a_i$.
Therefore,
\[
\begin{array}{rcl}
 (\mu_{B,M} \circ \sigma)(m) 
&=& 
(\mu_{B,M} \circ (\id_B \otimes \mu_{A,M})
\circ
(\sigma_{B,A} \otimes \id_M) 
\circ 
\sigma_{A,M}) (m) \\
&=&
\sum_{i=1}^p (\mu_{B,M} \circ 
(\id_B \otimes \mu_{A,M}) \circ
(\sigma_{B,A} \otimes \id_M))
(a_i \otimes m_i) \\
&=&
\sum_{i=1}^p \sum_{j=1}^{q_i}
(\mu_{B,M} \circ 
(\id_B \otimes \mu_{A,M}))
(b_j^{(i)}\otimes a_j^{(i)} \otimes m_i) \\
&=& 
\sum_{i=1}^p \sum_{j=1}^{q_i}
\mu_{B,M}( b_j^{(i)} \otimes a_j^{(i)} m_i) \\
&=&
\sum_{i=1}^p \sum_{j=1}^{q_i} 
f(b_j^{(i)})  a_j^{(i)} m_i = 
\sum_{i=1}^p a_i m_i = m,
\end{array}
\]
showing that 
$\mu_{B,M} \circ \sigma = \id_M$.
\end{proof}

Let $f : B \to A$ be a ring homomorphism.
Consider $A$ as a $B$-bimodule via 
$f$, that is 
$b \cdot a := f(b)a$ and 
$a \cdot b := a f(b)$, 
for $a \in A$ and $b \in B$.
Recall from Section \ref{sec:introduction}
that $A/B$ is called
separable 
if the multiplication map 
$\mu_{A/B} : A \otimes_B A \to A$,
defined by
the additive extension of
$\mu(a \otimes a') = aa'$, for $a,a' \in A$,
has a section in ${}_A \Mod_A$,
that is if
there is an $A$-bimodule map
$\sigma : A \to A \otimes_B A$ such that
$\mu_{A/B} \circ \sigma = \id_A$. 

\subsubsection*{Proof of 
Theorem \ref{thm:main3}}
First we show the ``only if'' statement. 
Let $\Res_f$ be separable.
We wish to show that $A/B$ is separable.
Put $\sigma_{B,A} := \mu_{B,A}^{-1}$.
Let $\sigma' : A \to A \otimes_B A$
be the morphism in ${}_B \Mod_A$ defined by
$\sigma' := 
(f \otimes \id_A) \circ \sigma_{B,A}$. 
Then 
\begin{equation}\label{eq:identity}
\Res_f(\mu_{A/B}) \circ \sigma'= \id_A
\end{equation}
as morphisms in ${}_B \FMod$. Indeed,
take $a \in A$. Then
$\sigma_{B,A}(a) = 
\sum_{i=1}^n b_i \otimes a_i$
for some $n \in \N$, $b_i \in B$ and $a_i \in A$ with 
$\sum_{i=1}^n f(b_i) a_i = a$. Therefore
\[
\begin{array}{rcl}
(\Res_f(\mu_{A/B}) \circ \sigma')(a) 
&=& 
(\Res_f(\mu_{A/B}) \circ  
(f \otimes \id_A) \circ \sigma_{B,A})(a)  \\
&=& 
\sum_{i=1}^n (\Res_f(\mu_{A/B}) \circ  
(f \otimes \id_A)) (b_i \otimes a_i) \\
&=&
\sum_{i=1}^n 
\Res_f(\mu_{A/B}) (f(b_i) \otimes a_i) 
= \sum_{i=1}^n f(b_i) a_i = a.
\end{array}
\]
Since $A$ is firm, both ${}_A A$ and ${}_A(A \otimes_B A)$ belong to ${}_A\FMod$, so the map
$\R^{\Res_f}_{A,\;A\otimes_B A}$ is well defined.
Let $\sigma : A \to A \otimes_B A$ be the morphism in 
${}_A\Mod$ defined by
$\sigma := \R^{\Res_f}_{A,\;A\otimes_B A}(\sigma')$.
By (SF3) and \eqref{eq:identity},
$\mu_{A/B} \circ \sigma = \id_A$.
What remains to show is that $\sigma$ 
is a morphism in $\Mod_A$. 
To this end, take $a \in A$ and 
let $\alpha_a : A \to A$ and
$\beta_a : A \otimes_B A \rightarrow A \otimes_B A$ denote 
the left $B$-linear maps given by right 
multiplication by $a$.
Since $\sigma'$ is a morphism in ${}_B \Mod_A$, 
%left $B$-linear 
it follows that $\sigma' \circ \alpha_a = \beta_a \circ \sigma'$ as morphisms in ${}_B \Mod$.
%left $B$-linear maps.
By (SF3), 
$\sigma \circ \alpha_a = \beta_a \circ \sigma$ as left $A$-linear maps,
since $\Res_f( \alpha_a ) = \alpha_a$ and 
$\Res_f( \beta_a ) = \beta_a$.
Therefore, $\sigma$ is right $A$-linear.

Now we show the ``if'' statement.
Suppose that $\mu_{A/B}$ has a
section $\sigma$ in ${}_A \Mod_A$.
Let $M$ and $N$ be firm left $A$-modules.
Set $\sigma_{A,M} := \mu_{A,M}^{-1}$.
Define 
\[
\R_{M,N}^{\Res_f} : 
\Hom_B(\Res_f(M),\Res_f(N)) \to 
\Hom_A(M,N)
\]
in the following way. 
Let $i_N : A \otimes_B N \to A \otimes_A N$ be the 
map $a \otimes_B n \mapsto a \otimes_A n$.
Take $g \in \Hom_B( \Res_f(M) , \Res_f(N) )$. Put
\[
\R^{\Res_f}_{M,N}(g) :=
\mu_{A,N} \circ i_N \circ (\id_A \otimes g) \circ 
(\id_A \otimes \mu_{A,M}) \circ 
(\sigma \otimes \id_M) \circ \sigma_{A,M}.
\]
Then $\R_{M,N}^{\Res_f}(g)$, being 
defined as a 
composition of morphisms in ${}_A \Mod$, 
is again a morphism in ${}_A \Mod$.
To facilitate the rest of the proof, we now
describe $\R_{M,N}^{\Res_f}(g)$ elementwise. 
Take $m \in M$, $n \in \N$, $a_i \in A$ and $m_i \in M$
with $\sum_{i=1}^n a_i m_i = m$. For each $i$,
take $n_i \in \N$ and $a_{ij},a_{ij}' \in A$
such that 
$\sigma(a_i) = \sum_{j=1}^{n_i} 
a_{ij} \otimes a_{ij}'$. Then
\[
\begin{array}{cl}
  & \R_{M,N}^{\Res_f}(g)(m) \\
= & (\mu_{A,N} \circ i_N \circ (\id_A \otimes g) 
\circ (\id_A \otimes \mu_{A,M}) \circ 
(\sigma \otimes \id_M) \circ 
\sigma_{A,M})(m) \\
= & 
\sum_{i=1}^n (\mu_{A,N} \circ i_N \circ 
(\id_A \otimes g) \circ 
(\id_A \otimes \mu_{A,M}) \circ 
(\sigma \otimes \id_M))
( a_i \otimes m_i ) \\
= & 
\sum_{i=1}^n \sum_{j=1}^{n_i}
(\mu_{A,N} \circ i_N \circ (\id_A \otimes g) 
\circ (\id_A \otimes \mu_{A,M})) 
( a_{ij} \otimes a_{ij}' \otimes m_i ) \\
= & 
\sum_{i=1}^n \sum_{j=1}^{n_i}
(\mu_{A,N} \circ i_N \circ (\id_A \otimes g))
( a_{ij} \otimes a_{ij}' m_i ) \\
= & \sum_{i=1}^n \sum_{j=1}^{n_i}
\mu_{A,N} ( a_{ij} \otimes g(a_{ij}' m_i) ) 
= \sum_{i=1}^n \sum_{j=1}^{n_i}
a_{ij} g(a_{ij}' m_i).
\end{array}
\]
Now we show (SF1). Suppose that $g$ is a morphism
in ${}_A \Mod$. Then
\[
\begin{array}{rcl}
 \R_{M,N}^{\Res_f}(g)(m) 
 &=& 
 \sum_{i=1}^n \sum_{j=1}^{n_i}
a_{ij} g(a_{ij}' m_i)  
= 
g \left( \sum_{i=1}^n \sum_{j=1}^{n_i}
a_{ij} a_{ij}' m_i \right) \\
&=& g \left( \sum_{i=1}^n \sum_{j=1}^{n_i}
\mu_{A/B}(a_{ij} \otimes a_{ij}') m_i \right) \\
&=&  g \left( \sum_{i=1}^n 
\mu_{A/B} \left( \sum_{j=1}^{n_i} a_{ij} \otimes a_{ij}' \right)
m_i \right) \\
&=&  g \left( \sum_{i=1}^n 
\mu_{A/B}( \sigma(a_i) ) m_i \right) 
= g \left( \sum_{i=1}^n a_i m_i \right) 
= g(m).
\end{array}
\]
Finally, we show (SF3). 
Let $M,N,P$ be modules in ${}_A \FMod$. 
Let  $g \in \Hom_B(N,P)$ and 
$h \in \Hom_B(M,N)$. Take $m \in M$.
We consider two cases.

Case 1: $g \in \Res_f(\Hom_A(N,P))$,
that is $g$ is $A$-linear. Then
\[
\begin{array}{rcl}
\R_{M,P}^{\Res_f}(g \circ h)(m)    
&=& \sum_{i=1}^n \sum_{j=1}^{n_i}
a_{ij} g( h(a_{ij}' m_i))  
= g \left( \sum_{i=1}^n \sum_{j=1}^{n_i}
a_{ij} h(a_{ij}' m_i) \right) \\
&=& g \left( \R_{M,N}^{\Res_f}(h)(m) \right)
= \left( \R_{N,P}^{\Res_f}(g) \circ 
\R_{M,N}^{\Res_f}(h) \right)(m).
\end{array}
\]

Case 2: $h \in \Res_f(\Hom_A(M,N))$,
that is $h$ is $A$-linear.
Since $\sum_{i=1}^n a_i m_i = m$, 
it follows that 
$\sum_{i=1}^n a_i h(m_i) = h(m)$. Hence
\[
\begin{array}{rcl}
\R_{M,P}^{\Res_f}(g \circ h)(m)    
&=& \sum_{i=1}^n \sum_{j=1}^{n_i}
a_{ij} g( h(a_{ij}' m_i))  
= \sum_{i=1}^n \sum_{j=1}^{n_i}
a_{ij} g(a_{ij}' h(m_i))  \\
&=& \R_{N,P}^{\Res_f}(g)(h(m)) 
= \left( \R_{N,P}^{\Res_f}(g) \circ 
\R_{M,N}^{\Res_f}(h) \right)(m).
\end{array}
\]
This establishes (SF3).
By Proposition \ref{prop:newsep},
$\Res_f$ is separable. \qed

\begin{prop}
Let $f : B \to A$ and $g : C \to B$ be left firm ring homomorphisms.
Suppose that $A$ and $B$ are firm rings.
\begin{itemize}[widest=(a)]

\item[{\rm (a)}]
$A/B$ separable and $B/C$ separable
$\Longrightarrow$
$A/C$ separable for $f \circ g : C \to A$.

\item[{\rm (b)}]
$A/C$ separable for $f \circ g : C \to A$
$\Longrightarrow$ $A/B$ separable.

\end{itemize}
\end{prop}

\begin{proof}
This follows immediately from 
Theorem \ref{thm:main3} and
Proposition \ref{prop:nastasescu}. 
\end{proof}

Suppose that \(f : B \to A\) is a ring 
homomorphism.  
The induction functor  
$\Ind_f : {}_B \Mod \to {}_A \Mod$
assigns to any left \(B\)-module \(N\) the left \(A\)-module \(A \otimes_B N\),
where \(A\) acts on \(A \otimes_B N\) by left multiplication
and the right \(B\)-action on \(A\) is given via \(f\).

\begin{prop}
Let $f : B \to A$ be a ring homomorphism. 

\begin{itemize}[widest=(a)]

\item[{\rm (a)}] $f$ left unital
$\Longrightarrow$
$\Ind_f : {}_B \Mod \to {}_A \Mod^1$;

\item[{\rm (b)}] $f$ left s-unital 
$\Longrightarrow$
$\Ind_f : {}_B \Mod \to {}_A \SMod$;

\item[{\rm (c)}] $f$ left unitary 
$\Longrightarrow$
$\Ind_f : {}_B \Mod \to {}_A \UMod$;

\item[{\rm (d)}] $f$ right firm 
$\Longrightarrow$
$\Ind_f : {}_B \FMod \to {}_A \FMod$.

\end{itemize}
\end{prop}

\begin{proof}
Let $N$ be a left $B$-module.
First we show (a), (b) and (c) 
simultaneously. Let $f$ be left
unital (respectively s-unital, unitary). 
Then $A$,
considered as a left $B$-module, is
unital (respectively s-unital, unitary).
Hence, the same conclusion holds for 
$A$ viewed as a left module over itself.
Therefore, the left $A$-module
$A \otimes_B N$ is also unital 
(respectively s-unital, unitary).

(d) Suppose that $f$ is right firm and $N$ 
is firm in ${}_B \FMod$. 
Put $M := A \otimes_B N$.
By Lemma \ref{lem:Mfirmifsection}, 
we need to show that $\mu_{A,M}$ 
has a section $\sigma$ in ${}_A \Mod$. 
Set $\sigma_{B,N} := \mu_{B,N}^{-1}$.
Let $i_M : A \otimes_B M \to A \otimes_A M$ be the 
map $a \otimes_B m \mapsto a \otimes_A m$.
Define $\sigma : M \to A \otimes_A M$ by
$\sigma := i_M \circ (\id_A \otimes f \otimes \id_N) \circ
(\id_A \otimes \sigma_{B,N})$.
Then $\sigma$, being a 
composition of morphisms in ${}_A \Mod$, is 
again a morphism in ${}_A \Mod$.
Now we show that
$\mu_{A,M} \circ \sigma = \id_M$.
Take $a \in A$ and $n \in N$. Take 
$k \in \N$, $b_i \in B$ and $n_i \in N$ 
with $\sigma_{B,N}(n) = \sum_{i=1}^k b_i \otimes n_i$, 
so that $\sum_{i=1}^k b_i n_i = 
(\mu_{B,N} \circ \sigma_{B,N})(n) = n$. Then
\[
\begin{array}{rcl}
(\mu_{A,M} \circ \sigma)( 
 a \otimes n ) 
 &=&  
(\mu_{A,M} \circ i_M \circ 
(\id_A \otimes f \otimes \id_N) \circ 
(\id_A \otimes \sigma_{B,N})) (a \otimes n) 
\\
&=& \sum_{i=1}^k (\mu_{A,M} \circ i_M \circ 
(\id_A \otimes f \otimes \id_N)) (a \otimes b_i \otimes n_i)
\\
&=&
\sum_{i=1}^k (\mu_{A,M} \circ i_M)
( a \otimes f(b_i) \otimes n_i )
= \sum_{i=1}^k  af(b_i) \otimes n_i
\\
&=& \sum_{i=1}^k a \otimes b_i n_i = 
a \otimes \sum_{i=1}^k b_i n_i = a \otimes n,
\end{array}
\]
showing that $\mu_{A,M} \circ \sigma = \id_M$.
\end{proof}

A ring homomorphism $f : B \to A$ is called firm if it 
is both left and right firm.

\begin{thm}
Let $B$ be a firm ring and suppose that 
$f : B \to A$ is a firm 
ring homomorphism. Then the functor
$\Ind_f : {}_B \FMod \to {}_A \FMod$ is 
separable if and only if 
$f$ is a split monomorphism in the category 
of $B$-bimodules. 
\end{thm}

\begin{proof}
First we show the ``only if'' statement.
Suppose that $\Ind_f$ is separable.
We wish to show that $f$ is a split monomorphism
in the category of $B$-bimodules.
Set $\sigma_{A,B} := \mu_{A,B}^{-1}$ 
and $\gamma := (\mu_{A/B} \otimes \id_B) \circ 
(\id_A \otimes \sigma_{A,B})$. 
Then $\gamma$ is 
a morphism in ${}_A \Mod_B$. 
Now we show that
\begin{equation}\label{eq:ABgammaAB}
\mu_{A,B} \circ \gamma \circ 
\Ind_f(f) = 
\mu_{A,B} \circ 
\Ind_f(\id_B). 
\end{equation}
Suppose that 
$a \in A$ and $b \in B$.
Take $n \in \N$, $a_i \in A$ and $b_i \in B$
such that $\sigma_{A,B}( f(b) ) = 
\sum_{i=1}^n a_i \otimes b_i$.
Then $\sum_{i=1}^n a_i f(b_i) = 
\mu_{A,B}(\sigma_{A,B}(f(b))) = f(b)$.
Thus,
\[
\begin{array}{rcl}
(\mu_{A,B} \circ \gamma \circ \Ind_f(f))
(a \otimes b) 
&=& 
\sum_{i=1}^n 
(\mu_{A,B} \circ (\mu_{A/B} \otimes \id_B))
(a \otimes a_i \otimes b_i) \\ 
&=&
\sum_{i=1}^n \mu_{A,B}( a a_i \otimes b_i ) 
= \sum_{i=1}^n a a_i f(b_i) \\
&=& 
a f(b) = 
(\mu_{A,B} \circ \id_{A \otimes_B B})(a \otimes b).
\end{array}
\]
This establishes \eqref{eq:ABgammaAB}.
Since $\mu_{A,B}$ is bijective, we get 
$\gamma \circ \Ind_f(f) = 
\Ind_f(\id_B)$.
By (SF3), 
$\R_{A,B}^{\Ind_f}(\gamma) \circ f = \id_B$.
Thus, $f$ splits as an additive map.
Since $\gamma$ is left $A$-linear,
$\R_{A,B}^{\Ind_f}(\gamma)$ is 
left $B$-linear. Now we show that 
$\R_{A,B}^{\Ind_f}(\gamma)$
is right $B$-linear.
To this end, take $b \in B$ and let
$m_b^B : B \to B$ and $m_b^A : A \to A$
denote the left $B$-linear maps given by right multiplication by $b$ and $f(b)$, respectively.
Since $\gamma$ is right $B$-linear, we have
$\gamma \circ \Ind_f(m_b^A)=\Ind_f(m_b^B)\circ \gamma$.
By (SF3), $\R_{A,B}^{\Ind_f}(\gamma)\circ m_b^A
= m_b^B\circ \R_{A,B}^{\Ind_f}(\gamma)$.
Therefore, $\R_{A,B}^{\Ind_f}(\gamma)$ is right $B$-linear.

Now we show the ``if'' statement.
Suppose that $f$ is a split monomorphism in the 
category of $B$-bimodules. Then there is a 
$B$-bimodule map
$s : A \to B$ such that $s \circ f = \id_B$.
Suppose that $M$ and $N$ are firm left $B$-modules.
Define 
$\R_{M,N}^{\Ind_f} : 
\Hom_A(\Ind_f(M),\Ind_f(N)) \to 
\Hom_B(M,N)$
in the following way. 
Take 
$g \in \Hom_A(\Ind_f(M) , \Ind_f(N))$.
Put
\[
\R_{M,N}^{\Ind_f}(g) :=
\mu_{B,N} \circ (s \otimes \id_N) 
\circ g \circ 
(f \otimes \id_M) \circ \mu_{B,M}^{-1}.
\]
Then $\R_{M,N}^{\Ind_f}(g)$
is a morphism in ${}_B \Mod$.
Now we show (SF1).
Take $m \in M$,
$k \in \N$, $b_i \in B$ and $m_i \in M$
with $\sum_{i=1}^k b_i m_i = m$. 
Let $g = \Ind_f(g')$ for some 
$g' \in \Hom_B(M,N)$. Then 
\[
\begin{array}{rcl}
\R_{M,N}^{\Ind_f}(g)(m) 
&=& 
\sum_{i=1}^k (\mu_{B,N} \circ
(s \otimes \id_N) \circ g)
(f(b_i) \otimes m_i) \\
&=&
\sum_{i=1}^k
(\mu_{B,N} \circ (s \otimes \id_N)
\circ (\id_A \otimes g')) 
(f(b_i) \otimes m_i) \\ 
&=& 
\sum_{i=1}^k
(\mu_{B,N} \circ (s \otimes \id_N))
(f(b_i) \otimes g'(m_i)) \\ 
&=& 
\sum_{i=1}^k
\mu_{B,N} (b_i \otimes g'(m_i)) \\
&=& \sum_{i=1}^k b_i g'(m_i) 
= g' \left( \sum_{i=1}^k b_i m_i \right)
= g'(m).
\end{array}
\]
Now we show (SF3).
Suppose that $M,N,P$ are modules in ${}_B \FMod$. 
Take $g \in \Hom_A(A \otimes_B N,
A \otimes_B P)$ and 
$h \in \Hom_A(A \otimes_B M,
A \otimes_B N)$. 
We consider two cases.

Case 1: $g \in \Ind_f(\Hom_B(N,P))$
that is 
$g = \id_A \otimes g'$ for 
some $g' \in \Hom_B(N,P)$. Then
\[
(s \otimes \id_P) \circ g \circ 
(f \otimes \id_N) \circ 
(s \otimes \id_N) = 
(\id_B \otimes g') \circ 
(s \otimes \id_N) = 
(s \otimes \id_P) \circ g.
\]
Therefore,
{\small \[
\begin{array}{cl}
 & \R_{N,P}^{\Ind_f}(g) \circ
\R_{M,N}^{\Ind_f}(h) \\ 
= & \mu_{B,P} \circ (s \otimes \id_P)
\circ g \circ (f \otimes \id_N)
\circ \mu_{B,N}^{-1} \, \circ 
 \mu_{B,N} \circ (s \otimes \id_N)
\circ h \circ (f \otimes \id_M) 
\circ \mu_{B,M}^{-1} \\
=& \mu_{B,P} \circ (s \otimes \id_P)
\circ g \circ (f \otimes \id_N) \, \circ   
(s \otimes \id_N) \circ h \circ 
(f \otimes \id_M) \circ \mu_{B,M}^{-1} \\
=& \mu_{B,P} \circ (s \otimes \id_P) 
\circ g \circ h \circ 
(f \otimes \id_M) \circ \mu_{B,M}^{-1} 
= \R_{M,P}^{\Ind_f}(g \circ h).
\end{array}
\]}

Case 2: $h \in \Ind_f(\Hom_B(M,N))$
that is $h = \id_A \otimes h'$ for 
some $h' \in \Hom_B(M,N)$. Then
\[
(f \otimes \id_N) \circ
(s \otimes \id_N) \circ h \circ 
(f \otimes \id_M)  = 
(\id_A \otimes h') \circ 
(s \otimes \id_M) = 
h \circ (f \otimes \id_M).
\]
Therefore,
{\small \[
\begin{array}{cl}
 & \R_{N,P}^{\Ind_f}(g) \circ
\R_{M,N}^{\Ind_f}(h) \\
=& \mu_{B,P} \circ (s \otimes \id_P)
\circ g \circ (f \otimes \id_N)
\circ \mu_{B,N}^{-1} \, \circ 
\mu_{B,N} \circ (s \otimes \id_N)
\circ h \circ (f \otimes \id_M) 
\circ \mu_{B,M}^{-1} \\
=& \mu_{B,P}  \circ  (s \otimes \id_P)
\circ g \circ 
(f \otimes \id_N) \, \circ 
(s \otimes \id_N)  \circ h \circ
(f \otimes \id_M)  \circ \mu_{B,M}^{-1} \\
=& \mu_{B,P} \circ (s \otimes \id_P) 
\circ g \circ h \circ 
(f \otimes \id_M) \circ \mu_{B,M}^{-1} 
= \R_{M,P}^{\Ind_f}(g \circ h).
\end{array}
\]}
This establishes (SF3).
By Proposition \ref{prop:newsep},
$\Ind_f$ is separable.
\end{proof}

\section{Semisimple modules}\label{sec:semisimplemodules}

From now on, $A$ denotes 
a fixed left s-unital ring, and we 
consider the category ${}_A \FMod$
of firm left $A$-modules. Recall from 
Proposition~\ref{prop:sunitalimpliesfirm} that 
\begin{equation}\label{eq:threeequal}
{}_A \FMod = 
{}_A \UMod = {}_A \SMod.
\end{equation}

Suppose that $M$ and $P$ are modules 
in ${}_A \FMod$
with $P \subseteq M$. 
Then $P$ is called 
a submodule of $M$ in ${}_A \FMod$. 
By \eqref{eq:threeequal}, the  
quotient module $M/P$ in ${}_A \Mod$  
also belongs to ${}_A \FMod$.
For a morphism $f : M \to N$ 
in ${}_A \FMod$, we define its kernel 
$\Ker(f) := \{ m \in M \mid f(m) = 0 \}$
and its image $\Ima(f) := \{ f(m) \mid m \in M \}$.
By \eqref{eq:threeequal},
$\Ker(f)$ is a submodule of $M$ 
in ${}_A\FMod$ 
and $\Ima(f)$ is a submodule of $N$ in ${}_A\FMod$.

\begin{prop}\label{prop:monoisinj}
Suppose that $f : M \to N$ is a morphism 
in ${}_A \FMod$.

\begin{itemize}[widest=(a)]

\item[{\rm (a)}] $f$ is a monomorphism
$\Longleftrightarrow$ $f$ is injective;

\item[{\rm (b)}] $f$ is an epimorphism
$\Longleftrightarrow$ $f$ is surjective;

\item[{\rm (c)}] $f$ is an isomorphism
$\Longleftrightarrow$ $f$ is bijective.

\end{itemize}
\end{prop}

\begin{proof}
(a) The implication $(\Leftarrow)$ is trivial.
Now we show $(\Rightarrow)$. 
Suppoe that $f$ is a monomorphism. 
Define $g,h : \Ker(f) \to M$ by 
$g(m)=m$ and $h(m)=0$ for $m \in \Ker(f)$.
Then $f g = 0 = fh$. 
Since $f$ is a monomorphism, $g=h$. Hence,
$\Ker(f) = g( \Ker(f) ) = h( \Ker(f) ) = \{ 0 \}$.

(b) The implication $(\Leftarrow)$ is trivial.
Now we show $(\Rightarrow)$. 
Let $f$ be
an epimorphism. Define $g,h : N \to N/\Ima(f)$
by letting $g$ be the quotient map and $h$
the zero map. Then $gf = 0 = hf$. Since $f$ is 
an epimorphism, $g = h$, so that 
$\Ima(f)=N$. 

(c) This follows from (a) and (b).
\end{proof}

The next result shows that subobjects and 
quotient objects
in ${}_A\FMod$ correspond 
to submodules
and quotient modules in ${}_A\FMod$,
respectively.

\begin{prop}\label{prop:bijections}
Let $M \in  {}_A \FMod$. The following maps
are bijections:
\[
\begin{array}{c}
\gamma : \Sub_{{}_A\FMod}(M) \longrightarrow 
\{ \text{submodules of } M \},
\quad 
\gamma([f:P \to M]) = f(P), \\[5pt]
\delta : \Quot_{ {}_A\FMod }(M) \longrightarrow 
\{ \text{quotient modules of } M \},
\quad 
\delta([q:M \to Q]) = Q.
\end{array}
\]
\end{prop}

\begin{proof}
The claim for $\gamma$ follows from 
\cite[Thm. 4.2]{valjako2020}.
Now we show the part for $\delta$.

If $[q]=[q']$ in $\Quot(M)$, then $q'=\varphi q$ 
for some isomorphism $\varphi : Q \to Q'$, so 
$Q\cong Q'$, showing that $\delta$ is well defined. 
If $\delta([q])=\delta([q'])$, then $Q \cong Q'$ 
and choosing an isomorphism $\varphi:Q\to Q'$ yields $q'=\varphi\circ q$, so $[q]=[q']$, proving injectivity. For any quotient module $M/N$, the canonical map 
$p:M\to M/N$ is an epimorphism and $\delta([p])=M/N$,
proving surjectivity.
\end{proof}

Let $M$ be a module in ${}_A \FMod$.
We say that $M$ is simple 
in ${}_A\FMod$
if $M$ is nonzero, and $M$ has no 
submodules in ${}_A \FMod$
except the zero module and itself.

\begin{prop}
Let $M \in {}_A \FMod$.
The following assertions are equivalent:
\begin{itemize}[widest=(iii)]

\item[{\rm (i)}] $M$ is simple 
in ${}_A\FMod$;

\item[{\rm (ii)}] $M$ is subobject simple
in ${}_A\FMod$;

\item[{\rm (iii)}] $M$ is quotient simple
in ${}_A\FMod$.

\end{itemize}
\end{prop}

\begin{proof}
By Proposition~\ref{prop:bijections}, 
categorical subobjects of $M$ are in bijection with 
submodules of $M$. Likewise, categorical quotient 
objects of $M$ are in bijection with quotient modules 
of $M$. Therefore $\Sub_{{}_A\FMod(M)}$ has exactly 
two elements if and only if $M$ has exactly two submodules, 
and $\Quot_{{}_A\FMod(M)}$ has exactly two elements 
if and only if $M$ has exactly two quotient modules, 
and these conditions are equivalent since the lattice 
of submodules and the lattice of quotient modules of $M$ 
are anti-isomorphic via $N \mapsto M/N$. 
\end{proof}

\begin{lem}\label{lem:simplesubmodule}
Suppose that $M$ is a subobject semisimple
module in ${}_A \FMod$ and $N$ is a nonzero 
submodule of $M$. Then $N$ contains a 
simple submodule.
\end{lem}

\begin{proof}
Take a nonzero $n \in N$. By Zorn's lemma, there is 
a maximal submodule $P$ of $N$ with $n \notin P$.
Since $P \subseteq M$ and $M$ is subobject semisimple in ${}_A\FMod$, the inclusion $P \hookrightarrow M$ splits. Hence
$M = P \oplus P'$
for some submodule $P'$ of $M$. Put
$Q := N \cap P'$.
Then $N = P \oplus Q$.
Since $n \in N \setminus P$, we have $Q \neq 0$.
We claim that $Q$ is simple. Let $\{ 0 \} \neq Q' \subseteq Q$ 
be a nonzero submodule. Since $Q' \subseteq M$ and $M$ is subobject semisimple, the inclusion $Q' \hookrightarrow M$ splits. In particular, $Q = Q' \oplus Q''$
for some submodule $Q''$ of $Q$. Hence
$N = P \oplus Q' \oplus Q''$.
If $Q'' \neq 0$, then both $P \oplus Q'$ and $P \oplus Q''$
properly contain $P$, so by maximality of $P$ 
both contain $n$. Thus
$n \in (P \oplus Q') \cap (P \oplus Q'') = P$,
a contradiction. Therefore $Q''=0$, so $Q'=Q$. 
Thus $Q$ is simple.
\end{proof}

Let $M$ be a module in ${}_A \FMod$.
We say that $M$ is semisimple 
in ${}_A\FMod$
if $M$ is the direct sum of simple 
submodules in ${}_A\FMod$.

\begin{prop}\label{prop:semisimpleequiv}
Let $M \in {}_A \FMod$. 
The following assertions are equivalent:
\begin{itemize}[widest=(iii)]

\item[{\rm (i)}] $M$ is semisimple 
in ${}_A\FMod$;

\item[{\rm (ii)}] $M$ is quotient semisimple
in ${}_A\FMod$;

\item[{\rm (iii)}] $M$ is subobject 
semisimple in ${}_A\FMod$.

\end{itemize}
\end{prop}

\begin{proof}
(i)$\Rightarrow$(ii):
Let $q : M \to N$ be an epimorphism. We consider 
two cases. 

Case 1: $\Ker(q)=M$. Then $N = \{ 0 \}$.
Let $s : N \to M$ be the zero map. 
Then $q \circ s = \id_N$.

Case 2: $\Ker(q) \subsetneq M$.
Since $M$ is semisimple, $M=\oplus_{i\in I} M_i$ 
for some simple submodules $M_i$ of $M$.
Let $S$ be the subset of nonzero submodules $L$
of $M$ such that $\Ker(q) \cap L = \{ 0 \}$.
Then $S$ is nonempty. Indeed,
by the assumptions, there is $i \in I$ with
$M_i \not \subseteq \Ker(q)$, so that 
$M_i \cap \Ker(q) = \{ 0 \}$, by simplicity of $M_i$.
The set $S$, ordered by inclusion, is inductive,
since the union of a chain in $S$ again belongs to $S$.
Thus, by Zorn's lemma, it has a maximal element $P$.
Seeking a contradiction, suppose that 
$\Ker(q) + P \subsetneq M$. Then there is $j \in I$
with $M_j \not \subseteq \Ker(q) + P$ so that 
$M_j \cap (\Ker(q) + P) = \{ 0 \}$ by simplicity 
of $M_j$. But then $P \subsetneq P + M_j$ and 
$\Ker(q) \cap (P + M_j) = \{ 0 \}$ which violates 
the maximality of $P$, which is a contradiction.
Therefore, $M = \Ker(q) \oplus P$.
The restriction $q|_P : P \to N$ is injective since 
$P \cap \Ker(q) = \{ 0 \}$, and surjective since 
$q(P) = q(\Ker(q)) + q(P) = q( \Ker(q) + P ) =
q(M) = N$, hence an isomorphism. 
Letting $s : N \to M$ be the composite of $(q|_P)^{-1}$ with the inclusion $P \hookrightarrow M$ yields 
$q \circ s=\id_N$. 
Hence, $M$ is quotient semisimple.

(ii)$\Rightarrow$(iii):
Suppose that $M$ is quotient
semisimple. Let $i:K \to M$ be a monomorphism.  
Consider the cokernel $p : M \to M/K$, which is an 
epimorphism. By hypothesis there is  
$s : M/K \to M$ with 
$p \circ s=\mathrm{id}_{M/K}$. 
For any $m\in M$ we then have 
$m- s(p(m)) \in K$, so we 
may define a map $r:M\to K$ by 
$r(m)=m-s(p(m))$. 
Since $p(m - s(p(m))) = p(m)-p(s(p(m))) = 
p(m) - p(m) = 0$,
we have $r(m) \in K$ for all $m \in M$.
Thus $r : M \to K$ is well defined. It is clearly 
$A$-linear. Moreover, for $k\in K$ we have 
$p(k)=0$, hence $r(k)=k$, so that 
$r \circ i=\mathrm{id}_K$.
Hence, $M$ is subobject semisimple.

(iii)$\Rightarrow$(i): 
Let $M$ be subobject semisimple.
If $M = \{0\}$, then $M$ is the direct sum of an 
empty family of simple submodules.
Assume now that $M \neq \{ 0 \}$.
By Lemma~\ref{lem:simplesubmodule} and Zorn's lemma,
there exists a nonempty family $\{N_i\}_{i\in I}$ of simple 
submodules such that $N = \oplus_{i\in I} N_i$
is maximal among direct sums of simple submodules of $M$.
Since $N\subseteq M$ and $M$ is subobject semisimple, we have
$M = N\oplus P$ for some submodule $P$ of $M$.
Suppose that $P \neq \{ 0 \}$.
By Lemma~\ref{lem:simplesubmodule}, $P$ contains a simple 
submodule $P'$.
Since $P'\subseteq M$ and $M$ is subobject semisimple, the inclusion $P'\hookrightarrow M$ splits.
Hence $M = P'\oplus X$ for some submodule $X$ of $M$.
Thus, we get
$P = P'\oplus (P\cap X)$.
Hence $M = N \oplus P'\oplus (P\cap X)$,
contradicting the maximality of the family $\{N_i\}_{i\in I}$.
Therefore $P = \{0\}$, so $M=\bigoplus_{i\in I} N_i$.
Hence $M$ is semisimple.
\end{proof}

\begin{prop}\label{prop:semisimpledirectsum}
Let $M \in {}_A\FMod$. The following 
assertions hold:
\begin{itemize}[widest=(a)]

\item[{\rm (a)}] The module $M$ is 
semisimple in ${}_A\FMod$
if and only if every submodule and 
every quotient
module of $M$ is semisimple in ${}_A\FMod$.

\item[{\rm (b)}] Suppose 
$M = \oplus_{i \in I} M_i$
for some submodules $(M_i)_{i \in I}$. 
Then $M$ is 
semisimple in ${}_A\FMod$ 
if and only if each $M_i$
is semisimple in ${}_A\FMod$.

\end{itemize}
\end{prop}

\begin{proof}
(a) The ``if'' part is trivial, since $M$ is both a submodule and a quotient module of itself.
Now suppose that $M$ is semisimple in ${}_A\FMod$.
Let $N$ be a submodule of $M$.
We first show that $N$ is semisimple.
Let $P$ be a submodule of $N$.
By Proposition~\ref{prop:semisimpleequiv}, $M$ is subobject semisimple, so there exists a submodule $Q$ of $M$ such that
$M = P\oplus Q$. Hence
$N = N\cap M = N\cap (P\oplus Q) = P\oplus (N\cap Q)$.
Thus every submodule of $N$ is a direct summand of $N$, so $N$ is subobject semisimple.
By Proposition~\ref{prop:semisimpleequiv}, $N$ is semisimple.
Now we show that $M/N$ is semisimple.
Again by Proposition~\ref{prop:semisimpleequiv}, 
there exists a submodule $R$ of $M$ such that
$M = N \oplus R$.
Since $R$ is a submodule of the semisimple module $M$, 
the first part shows that $R$ is semisimple.
Moreover, $M/N \cong R$. Hence $M/N$ is semisimple.

(b)
If $M$ is semisimple, then each $M_i$ is a submodule of $M$, 
so $M_i$ is semisimple by (a).
Conversely, suppose that each $M_i$ is semisimple.
For each $i\in I$, write
$M_i = \oplus_{j\in J_i} S_{ij}$,
where each $S_{ij}$ is simple. Then
$M = \oplus_{i\in I} M_i
= \oplus_{i\in I} \oplus_{j\in J_i} S_{ij}$,
so $M$ is a direct sum of simple submodules.
Hence $M$ is semisimple in ${}_A\FMod$.
\end{proof}

We say that $A$ is left semisimple 
if $A$ is semisimple in ${}_A \FMod$.

\begin{prop}\label{prop:Asemisimpleprojective}
The ring $A$ is left semisimple if and only if 
every module in ${}_A \FMod$ is semisimple.
\end{prop}

\begin{proof}
The ``if'' direction is immediate.
Now we show the ``only if'' direction.
Suppose that $A$ is left semisimple.
Let $M$ be a module in ${}_A \FMod$. 
Since $M$ is unitary, we have $AM=M$ so that
$M=\sum_{m \in M}Am$.
Let $A^{(M)}$ denote the direct sum 
of $|M|$ copies of $A$. 
Define a surjective $A$-linear map
$p : A^{(M)} \ni (a_m)_{m \in M} \mapsto 
\sum_{m \in M} a_m m \in M$.
By Proposition~\ref{prop:semisimpledirectsum}(b), 
$A^{(M)}$ is semisimple.
Therefore, by Proposition~\ref{prop:semisimpledirectsum}(a),
$M \cong A^{(M)}/\Ker(p)$ is semisimple.
\end{proof}

\subsubsection*{Proof of Theorem \ref{thm:main4}}
Let $A$ and $B$ be left s-unital rings
with $B$ left semisimple.
Suppose that $f : B \to A$ is a left s-unital 
ring homomorphism such that $A/B$ 
is separable. We wish to show that $A$ is left semisimple.
Let $M \in {}_A\FMod$.
By Proposition~\ref{prop:Asemisimpleprojective}, the module $\Res_f(M)$ is semisimple in ${}_B\FMod$.
By Proposition~\ref{prop:semisimpleequiv}, it is 
therefore subobject semisimple in ${}_B\FMod$.
By Theorem~\ref{thm:main3}, $\Res_f$ is separable.
Clearly, $\Res_f$ preserves monomorphisms.
By Proposition~\ref{prop:Fsemisimple}, $M$ is subobject semisimple in ${}_A\FMod$.
Hence, by Proposition~\ref{prop:semisimpleequiv}, $M$ is semisimple in ${}_A\FMod$.
Therefore, by Proposition~\ref{prop:Asemisimpleprojective}, $A$ is left semisimple. \qed

\subsubsection*{Proof of Theorem \ref{thm:main5}}
Let $G$ be a finite group 
of order $n:=|G|$.
Let $B$ be a ring with local units
satisfying $nB = B$. 
Let $A$ denote the group ring $B[G]$. 
We wish to show that $A/B$ is separable.
To this end, we will
construct an $A$-bimodule splitting $\sigma$ of the multiplication map
$\mu_{A/B} : A \otimes_B A \to A$.
Take $a \in A$ and an idempotent $e \in B$ such that $ae=a=ea$.
Since $nB=B$, there exists $b \in B$ such that $nb=e$.
Then $(ne)(ebe)=e(nb)e=e^2ee=e$ and, analogously,
$(ebe)(ne)=e$.
Hence $ne$ is invertible in the unital ring $eBe$.
Define
$\sigma(a):=\sum_{g\in G} ag \otimes (ne)^{-1}g^{-1}$.
Then
\[
\begin{aligned}
(\mu_{A/B}\circ \sigma)(a)
&=
\sum_{g\in G} ag(ne)^{-1}g^{-1} 
=
\sum_{g\in G} a(ne)^{-1}gg^{-1} \\
&=
\sum_{g\in G} a(ne)^{-1} 
=
n\,a(ne)^{-1} 
=
a(ne)(ne)^{-1} 
= ae = a.
\end{aligned}
\]
Therefore, $\sigma$ is a splitting of $\mu_{A/B}$.
Clearly, $\sigma$ is left $A$-linear.

Next we show that $\sigma$ is well defined.
Suppose that $e' \in B$ is another idempotent such that $ae'=a=e'a$.
Take an idempotent $f \in B$ such that
$ef=fe=e$ and $e'f=fe'=e'$.
Let $c=(nf)^{-1}\in fBf$, so that
$(nf)c=f$. Then
$(ne)c=e(nf)c=ef=e$, so
$ec=(ne)^{-1}$.
Since $c=(nf)^{-1}$, it follows that
$e(nf)^{-1}=(ne)^{-1}$.
Similarly, $e'(nf)^{-1}=(ne')^{-1}$.
Therefore,
\[
\begin{array}{l}
\displaystyle
\sum_{g \in G} ag \otimes (ne)^{-1}g^{-1}
=
\sum_{g \in G} ag \otimes e(nf)^{-1}g^{-1}
=
\sum_{g \in G} aeg \otimes (nf)^{-1}g^{-1}
\\[10pt]
\displaystyle
=
\sum_{g \in G} ae'g \otimes (nf)^{-1}g^{-1}
=
\sum_{g \in G} ag \otimes e'(nf)^{-1}g^{-1}
=
\sum_{g \in G} ag \otimes (ne')^{-1}g^{-1}.
\end{array}
\]

Finally, we show that $\sigma$ is right $A$-linear.
Take $b \in B$ and $h \in G$.
Choose an idempotent $e \in B$ such that
$ae=a=ea$ and $be=b=eb$. Since
$b(ne)=nb=(ne)b$, it follows that
$(ne)^{-1}b=b(ne)^{-1}$. Hence
\[
\begin{aligned}
\sigma(a)bh
&=
\sum_{g\in G} ag \otimes (ne)^{-1}g^{-1}bh 
=
\sum_{g\in G} ag \otimes b(ne)^{-1}g^{-1}h \\
&=
\sum_{p\in G} ahp \otimes b(ne)^{-1}p^{-1} 
=
\sum_{p\in G} abhp \otimes (ne)^{-1}p^{-1} 
=
\sigma(abh),
\end{aligned}
\]
where $p := h^{-1}g$.
Thus $\sigma$ is right $A$-linear.
Therefore, $A/B$ is separable.
\qed

\subsubsection*{Proof of Theorem \ref{thm:main6}}
Let $G$ be a finite group 
of order $|G|$.
Let $B$ be a semisimple ring with local units 
satisfying $|G|B = B$. 
Put $A := B[G]$. We wish to show that $A$ is 
left semisimple.
Since $B$ is locally unital,
the ring inclusion $B \to A$
is left s-unital. Therefore, by Theorem
\ref{thm:main4} and Theorem \ref{thm:main5},
$A$ is left semisimple.
\qed 

\begin{rem}
It is easy to produce examples of nonunital rings \(B\) 
satisfying the hypotheses of
Theorem~\ref{thm:main6} for any finite group \(G\).
Let \(K\) be a field of characteristic zero and let
\(B = K^{(\N)}\) be the ring from Example~\ref{ex:examplesrings}(a).
Then \(B\) is semisimple and nonunital, 
yet has local units, and satisfies
\(nB = B\) for every \(n \in \N\).
Hence Theorem~\ref{thm:main6} applies to the group ring \(B[G]\) for any finite group \(G\).

Moreover, one can even find simple nonunital rings
$B$ with local units such that 
\(nB = B\) for all \(n \in \N\).
A standard example is the ring 
$\mathrm{FM}_{\N}(K)$ of finitary $\N \times \N$ matrices over $K$, that is, the ring consisting of all matrices $(a_{ij})_{i,j \in \N}$ with $a_{ij} \in K$ such that $a_{ij}=0$ for all but finitely many pairs $(i,j) \in \N \times \N$, over 
a field $K$ of characteristic zero.
This ring has local units, since any finite set of 
matrices is supported on finitely many rows and columns, 
and the diagonal idempotent corresponding to those rows 
and columns acts as a local unit for that set.
We leave the simplicity
verification to the reader.
\end{rem}

Let $A$ be a left s-unital ring.
We say that $A$ is  
left hereditary if for every 
projective left $A$-module $M$ in
${}_A\FMod$, all $A$-submodules of $M$ 
are again projective.

\begin{thm}\label{thm:ABhereditary}
Let $A$ and $B$ be rings with $A$ s-unital and
$B$ left hereditary.
Let $f : B \to A$ be a left firm 
ring homomorphism with $A/B$ is separable.
Suppose that 
$\Res_f : {}_A \FMod \to {}_B \FMod$ 
preserves projective modules. 
Then $A$ is left hereditary.
\end{thm}

\begin{proof}
Let $M \in {}_A \FMod$ be projective and
let $N$ be a submodule of $M$.
By the assumptions,
$\Res_f(M)$ is projective in  ${}_B \FMod$.
Since $B$ is left hereditary, also
$\Res_f(N)$ is projective in  ${}_B \FMod$.
By Theorem \ref{thm:main3},
$\Res_f$ is separable.
Clearly, $\Res_f$ preserves epimorphisms. 
Thus, by 
Proposition~\ref{prop:Fseparableprojective},
$N$ is projective in ${}_A \FMod$.
Therefore, $A$ is left hereditary. 
\end{proof}

\noindent
\textbf{Acknowledgement.}
The author is grateful to the anonymous referee for many
 helpful comments and suggestions that has greatly improved the quality of the manuscript, and in particular for suggesting Lemma~\ref{lem:limitsreflect}.

\end{document}